\documentclass[10pt,a4paper]{amsart}
          
\usepackage{mystyle}

\title{  Local Euler obstructions of toric varieties}    
\author{Bernt Ivar Utst$\o$l N$\o$dland} \thanks{Author affiliation: Department of Mathematics, University of Oslo, Moltke Moes vei 35,
0851 Oslo, Norway Email: berntin@math.uio.no }

\date{\today}

\begin{document}

\begin{abstract}
We use Matsui and Takeuchi's formula for toric $A$-discriminants to give algorithms for computing local Euler obstructions and dual degrees of toric surfaces and $3$-folds. In particular, we consider  weighted projective spaces. As an application we give counterexamples to a conjecture by Matsui and Takeuchi. As another application we recover the well-known fact that the only defective normal toric surfaces are cones.
\end{abstract}

\maketitle

\smallskip
\noindent \textbf{Keywords.} toric varieties,
    weighted projective spaces,
singularities,
local Euler obstruction,
lattice polytopes,
continued fractions

\section{Introduction}

\footnotetext{\textit{Mathematics Subject Classification}: \textup{Primary:14M25 Secondary:52B20,14B05,14C17}}

The local Euler obstruction was used by MacPherson \cite{MacPherson} in his construction of Chern classes for singular varieties. For a variety $X$ the local Euler obstruction is a constructible function $\Eu: X \to \Z$ which takes the value $1$ at smooth points of $X$. It is related to the Chern--Mather class and to the Chern--Schwartz--MacPherson class of $X$ (see Remark \ref{chclass}).

Several equivalent definitions of the local Euler obstruction have been given, such as Kashiwara's definition of the local characteristic for a germ of an irreducible analytic space \cite{Kashiwara}. The first  algebraic formula was given by Gonz\'alez-Sprinberg and Verdier \cite{gonz}. Matsui and Takeuchi use a topological definition \cite{MatTak}, which defines the local Euler obstruction of $X$ inductively using the Whitney stratification of $X$. They use this definition to prove a formula for the local Euler obstruction on a (not necessarily normal) toric variety $X$. In this article we will apply this formula to compute the local Euler obstructions of toric varieties of dimension $\leq 3$.

For a normal toric surface $X$, we have that $X$ is smooth if and only if $\Eu(X)=\mathbbm{1}_X$ \cite[Cor 5.7]{MatTak}. \cite{MatTak} conjecture that the corresponding statement should also hold for a higher dimensional normal and projective toric variety. As an application we present counterexamples to this conjecture.

A motivation for studying Euler obstructions comes from formulas for the degrees of dual varieties. Given a projective variety variety $X \subset \p^N$, its dual variety $X^\vee \subset {\p^N}^\vee$ is the closure of the set of hyperplanes $H \in {\p^N}^\vee$ such that there exists a smooth point $x \in X$ with $T_xX \subset H$. Generally  $X^\vee$ will be a hypersurface in ${\p^N}^\vee$. Finding its equation is usually very difficult, but there are results which give the degree. Gelfand, Kapranov and Zelevinsky \cite{GKZ} proved a combinatorial formula for the degree of the dual variety of an embedded smooth toric variety. Matsui and Takeuchi generalized this formula to singular toric varieties, by weighting the terms by the local Euler obstruction. We will use this to describe algorithms to compute the degree of the dual variety of some toric varieties, in particular weighted projective spaces  of dimension $\leq 3$.

There has been  recent interest in the local Euler obstruction. Aluffi  studied Chern--Mather and Chern--Schwartz--MacPherson classes in \cite{Aluffi}. Helmer and Sturmfels studied polar degrees and the local Euler obstruction in \cite{SturmHelm} related to the problem of finding the Euclidean distance degree of a variety.  This problem is closely related to the contents of the current paper, since the Euclidean distance degree is expressible in terms of polar degrees, which in turn is expressible in terms of Matsui and Takeuchi's formulas involving the local Euler obstruction. In particular Helmer and Sturmfels study codimension one toric varieties \cite[Thm 3.7]{SturmHelm}, and also they briefly study surfaces.

In Section 2 we define the local Euler obstruction. We recall some basic facts about toric varieties.

In Section 3 we present Matsui and Takeuchi's method for computing the local Euler obstruction of toric varieties and the degree of dual varieties.

In Section 4 we introduce our main examples of study, the weigthed projective spaces. We describe them via toric geometry.

In Section 5 we follow Chapter 5 of \cite{Mork} and apply the theory to  toric surfaces. This  relates to Hirzebruch--Jung continued fractions and the minimal resolution of singularities. We then do explicit computations for weighted projective planes.

In Section 6 we consider the local Euler obstruction of toric 3-folds. We prove that for a toric 3-fold $X_{P \cap M}$ with isolated singularities, the local Euler obstruction is always greater than or equal to $1$. We find counterexamples to a conjecture by Matsui and Takeuchi \cite[p.2063]{MatTak}.

In Section 7 we apply the above to describe which toric surfaces are dual defective, and to say something about which $3$-dimensional weighted projective spaces are dual defective. 

In the appendix we collect some computations of the local Euler obstruction and degrees of dual varieties for some weighted projective spaces.

\section{The local Euler obstruction} \label{secEu}

Given a complex projective variety $X$ of dimension $d$, consider the (generalized) Grassmann variety ${\rm Grass}_d(\Omega_X^1)$ representing locally free rank $d$ quotients of $\Omega_X^1$. The Nash blowup $\tilde X$ of $X$ is the closure of the image of the morphism $X_{\rm sm}\to {\rm Grass}_d(\Omega_X^1)$. Let $\pi\colon \tilde X\to X$ denote the projection. The Nash sheaf $\tilde \Omega$ is the restriction to $\tilde X$ of the tautological rank $d$ sheaf on ${\rm Grass}_d(\Omega_X^1)$. There is a surjection $\pi^\ast \Omega^1_X \to \tilde{\Omega}$, and the Nash blowup is universal with respect to birational morphisms $f: Y \to X$ such that there is a locally free sheaf $\mathcal{F}$ of rank $d$ on $Y$ and a surjection $f^\ast \Omega^1_X \to \mathcal{F}$. Let $\tilde T$ denote the dual of $\tilde \Omega$.

The local Euler obstruction of a point $x \in X$ is the integer
\[ \Eu(x) = \int_{\pi^{-1}(x)} c(\tilde{T}|_{\pi^{-1}(x)}) \cap s(\pi^{-1}(x),\tilde{X}). \]
On the smooth locus of a variety the local Euler obstruction takes the value  $1$. It is a local invariant, thus we can compute it on an open affine cover. 

This is the usual algebraic definition, used by amongst others \cite[Ex. 4.2.9]{Fulton2}. When the ambient variety is clear we will simply write $\Eu$ for the local Euler obstruction, however if there are different ambient varieties  we sometimes write $\Eu_X$ for the local Euler obstruction on $X$.

\begin{remark} \label{chclass}
 The Chern--Mather class $c^M(X)$ of a variety $X$ is defined by
\[ c^M (X) = \pi_\ast ( c(\tilde{T}) \cap [\tilde{X}]). \]
There is an isomorphism $T$ from cycles on $X$ to constructible functions on $X$  given by $\sum n_i [V_i] (p) \mapsto \sum n_i \Eu_{V_i}(p)$. Letting $c_\ast$ be $c^M \circ T^{-1}$, we have that $c_\ast$ is the unique natural transformation from constructible functions on $X$ to the homology of $X$ such that on a non-singular $X$ we have that $c^{SM} (X)$ is the Poincare dual of the total Chern class of $X$ \cite[Thm 1]{MacPherson}. The Chern-Schwartz-MacPherson class $c^{SM} (X)$ is defined as $c_\ast ( \mathbbm{1}_X)$.
\end{remark}

\subsection{Definitions and notation for toric varieties}

We shall use the notation and definitions from \cite{Cox} for toric varieties. Let $T$ be the torus $(\C^\ast)^n$ and let M denote its character lattice $\Hom(T,\C^\ast) \simeq \Z^n$. The dual $\Hom_\Z(M,\Z)$ of $M$ we denote by $N$. Any normal toric variety is of the form $X_\Sigma$ for a fan $\Sigma \subset N_\R$, and has the open affine cover $ \{ U_\sigma | \sigma \in \Sigma \}$.

We will sometimes be interested in toric varieties which are not normal: For a finite set of lattice points $A \subset M$, we can associate the toric variety $X_A \subset \p^{\# A-1}$ by mapping the torus via the characters corresponding to the lattice points in $A$ and taking the Zariski closure. These varieties are not necessarily normal.

For a subset $S \subset M_\R$ we denote by $\Conv(S)$ the convex hull of the points of $S$. Setting $P=\Conv(A)$, we get a (possibly different) embedding $X_{P \cap M} \subset \p^{\# P \cap M -1}$. The variety $X_A$ is also the image of the projection from  $\p^{\# P \cap M -1}$ to  $\p^{\# A-1}$ given by forgetting the coordinates corresponding to $P \cap M \setminus A$.

To a lattice polytope $P$ we can also associate the normal toric variety $X_P$ which equals $X_{kP \cap M}$ for any $k \in \mathbb{N}$ such that $kP$ is very ample (this is equivalent to a certain divisor on $X_{kP \cap M}$ being very ample), thus it is independent of any specific embedding and not necessarily isomorphic to $X_{P \cap M}$. We have that $X_P$ is isomorphic to $X_{\Sigma_P}$, the toric variety associated to the normal fan $\Sigma_P \subset N_\R$ of P. 

If the polytope $P$ is itself very ample, we will sometimes, by abuse of notation, identify the abstract variety $X_P$ with the embedded variety $X_{P \cap M}$. For instance all $2$-dimensional polytopes are very ample.

\section{The local Euler obstruction of toric varieties}

Consider a toric variety $X_A$ associated to a finite set $A$ in $M \simeq \Z^n$.  We will use the formula for the local Euler obstruction of toric varieties proved in \cite[Ch. 4]{MatTak}. It is proved using an equivalent topological definition of the Euler obstruction, defined by induction on the codimension of the strata of a Whitney stratification of the variety. 

\begin{remark}
One can quite explicitly describe both the Nash blowup of a toric variety, and its normalization as a toric variety, see \cite{nash},\cite{PerezTeissier}. It would be interesting to prove Matsui and Takeuchi's formula for the local Euler obstruction directly from the algebraic definition, for instance if one could  describe the Nash sheaf as a module over the Cox ring of $X_A$.
\end{remark}

Let $P$ be the convex hull of $A$. We may assume this has dimension $n$. Then $P$ is a lattice polytope in $M$. For a toric variety $X_A$ there is a one-to-one correspondence bewtween faces of $P$ and orbits of $X_A$ by \cite[Prop. 1.9]{GKZ}. The local Euler obstruction is constant on each orbit, hence for a face $\Delta \preceq P$ we can denote by $\Eu(\Delta)$ the common value of the local Euler obstruction on the orbit corresponding to $\Delta$. Matsui and Takeuchi describe the Euler obstruction combinatorially by induction on the codimension of the faces of $P$.

For a face $\Delta$ of $P$, let $L(\Delta)$ be the smallest linear subspace in $M_{\mathbb{R}}$ containing $\Delta$. The dimension of $L(\Delta)$ is equal to $\dim \Delta$. We can also associate a lattice to $\Delta$: $M_\Delta$ is the lattice generated by $A \cap \Delta$ in $L(\Delta)$.

Given faces $\Delta_\alpha$ and $\Delta_\beta$ of $P$ such that $\Delta_\beta \preceq \Delta_\alpha$, we can associate a lattice $M_{\alpha,\beta} \defeq M_\alpha \cap L(\Delta_\beta)$. We have that $M_\beta \subseteq M_{\alpha,\beta}$, but they are not necessarily equal (see Examples \ref{ind1} and \ref{ind2}). They are however both of maximal rank in $L(\Delta_\beta)$ which motivates the following definition:

\begin{definition}
Given faces $\Delta_\alpha$ and $\Delta_\beta$ of $P$ such that $\Delta_\beta \preceq \Delta_\alpha$, we define the index $i(\Delta_\alpha, \Delta_\beta)$ to be $[M_{\alpha,\beta} : M_\beta]$.
\end{definition}

For two faces $\Delta_\alpha$ and $\Delta_\beta$ of $P$ such that $\Delta_\beta \preceq \Delta_\alpha$ we may, after a translation, assume that $0$ is a vertex of $\Delta_\beta$. We denote by $S_\alpha$ the semigroup generated by $A \cap \Delta_\alpha$ in $M_\alpha$. Let $S_{\alpha,\beta}$ denote the image of $S_\alpha$ in the quotient lattice $M_\alpha/ M_{\alpha,\beta}$.

\begin{definition} \label{RSV}
Given faces $\Delta_\alpha$ and $\Delta_\beta$ of $P$ such that $\Delta_\beta \preceq \Delta_\alpha$, we define the normalized relative subdiagram volume $RSV_{\mathbb{Z}}(\Delta_\alpha \Delta_\beta)$ of $\Delta_\alpha$ along $\Delta_\beta$ by
\[ \RSV_{\mathbb{Z}}(\Delta_\alpha,\Delta_\beta) =  \Vol ( S_{\alpha,\beta} \setminus \Theta_{\alpha,\beta}), \]
where $\Theta_{\alpha,\beta}$ is the convex hull of $S_\alpha/\Delta_\beta \cap  M_\alpha/ M_{\alpha,\beta} \setminus \{ 0 \}$ in $(M_\alpha/ M_{\alpha,\beta})_\R $. The volume is normalized with respect to the $(\dim \Delta_\alpha - \dim \Delta_\beta)$-dimensional lattice $M_\alpha/ M_{\alpha,\beta}$. If $\Delta_\alpha = \Delta_\beta$ we set $\RSV_{\mathbb{Z}}(\Delta_\alpha,\Delta_\beta) = 1$.
\end{definition}

\begin{corollary} \label{euler} \cite[Thm 4.7]{MatTak}\label{Euler}
The local Euler obstruction of $X_A$ is described as follows: The value $\Eu(\Delta_\beta)$ for a face $\Delta_\beta$ of $P$ is determined by induction on the codimension of the faces of $P$ by the following:
\[\Eu(P) = 1, \]
\[\Eu(\Delta_\beta) = \sum_{\Delta_\beta \lneqq \Delta_\alpha} (-1)^{\dim \Delta_\alpha -\dim \Delta_\beta -1} i(\Delta_\alpha, \Delta_\beta) \RSV_{\mathbb{Z}}(\Delta_\alpha,\Delta_\beta) \Eu(\Delta_\alpha) . \]
\end{corollary}

\begin{remark} \label{ChernMather}
By \cite[Th. 2]{PieneCM} the value of the local Euler obstruction at a torus orbit is the coefficient of the orbit in the the Chern--Mather class of $X$, i.e. $c^M(X_A) = \sum_{\Delta \preceq \Conv(A)} \Eu(\Delta) [\Delta]$.
\end{remark}

\begin{example} \label{ind1}
Let $A$ be the following lattice points in $M \simeq \Z^2$:
\[ \begin{bmatrix} 0 \\ 0 \end{bmatrix} \begin{bmatrix} 0 \\ 1 \end{bmatrix} \begin{bmatrix} 1 \\ 1 \end{bmatrix} \begin{bmatrix} 2 \\ 0 \end{bmatrix} \begin{bmatrix} 2 \\ 1 \end{bmatrix}.  \]
Let $e$ be the edge of $P = \Conv(A)$ generated by the vector $(1,0)$. Then $M_{P,e}$ is the lattice $\Z (1,0)$. However since $(1,0) \notin A$ we have that the lattice $M_e = \Z (2,0)$. Thus the index $i(P,e) = 2$.
\end{example}

In the example above we do not have that $A$ equals $M \cap \Conv(A)$. One might suspect that this is the only way to get a nontrivial index, however the following example show this to be wrong.

\begin {example} \label{ind2}
Let $Q$ be the $3$-dimensional polytope in $M \simeq \Z^4$ with lattice points
\[ \begin{bmatrix} 0 \\ 0 \\ 0 \\ 0\end{bmatrix} \begin{bmatrix} 1 \\ 0 \\ 0 \\ 0 \end{bmatrix} \begin{bmatrix} 0 \\ 1 \\ 0 \\ 0 \end{bmatrix} \begin{bmatrix} 1 \\ 1 \\ 2 \\ 0 \end{bmatrix}. \]
Assume $Q$ is the facet of a $4$-dimensional polytope $P$ cut out by setting the last coordinate equal to $0$. Furthermore assume that $P$ has enough lattice points so that $\Z P = M$. Then $M_{P,Q} = \Z^3$, but $M_Q = \Z^2 \oplus 2 \Z$, so $i(P,Q) = 2$.
\end{example}

Next we state Matsui and Takeuchi's formula for the degree of  the dual variety of a toric variety.

\begin{proposition}\cite[Cor 1.6]{MatTak} \label{dualformula} 
Assume $A$ is a finite subset of $M \simeq \Z^n$ such that $X_A^\vee$ is a hypersurface in ${\p^N}^\vee$. Setting $P = \Conv(A)$, we have
\[ \degree X_A^\vee = \sum_{Q \preceq P} (-1)^{\codim Q}(\dim Q +1)\Eu(Q) \Vol(Q) . \]
where $\Eu(Q)$ is the (constant) value of the local Euler obstruction on the torus orbit associated to $Q$, and $\Vol(Q)$ is the normalized volume of $Q$ with respect to the sublattice spanned by lattice points of $Q$.
\end{proposition}

\begin{remark} \label{defect}
By \cite[Thm 1.4]{MatTak} $X_A^\vee$ is a hypersurface if and only if the formula above yields a non-zero number.
\end{remark}

If we take $A$ to be all lattice points of a polytope, we can simplify some calculations:
\begin{lemma} \label{codim}
Assume  $A=\Conv(A) \cap M$. If $\dim \Delta_\alpha - \dim \Delta_\beta = 1$ then $\RSV_\Z(\Delta_\alpha,\Delta_\beta)=1$.
\end{lemma}
\begin{proof}
This follows almost by construction: The quotient lattice $M_\alpha / M_{\alpha,\beta}$ will be  isomorphic to $\Z$. Then $S_\alpha/\Delta_\beta$ must be generated by either $1$ or $-1$, thus it follows $\RSV_\Z(\Delta_\alpha,\Delta_\beta)=1$.
\end{proof}
\begin{corollary} \label{cod}
Assume  $A=\Conv(A) \cap M$. For any $(n-1)$-dimensional face $\Delta \lneqq P$ we have $\Eu(\Delta)=i(P,\Delta)$.
\end{corollary}

We  need the following well-known fact: Given set of linearly independent vectors $b_1,...,b_n \in M$ let 
\[ T(b_1,...,b_n) = \{ \sum_{i=1}^n c_ib_i | 0 \leq c_i < 1 \} \subseteq M_{\mathbb{R}} = M \otimes \mathbb{R} .\]
\begin{lemma} \label{basis}
The vectors $b_1,...,b_n$ form a basis for the lattice $M$ if and only if $T(b_1,...,b_n) \cap M = \{ 0 \}$ .
\end{lemma}

The following lemma will be useful when we study surfaces and $3$-folds:
\begin{lemma} \label{index1}
If $A$ is the set of lattice points of a convex lattice polytope of dimension $\leq 3$, then for any two faces  $\Delta_\beta \preceq \Delta_\alpha \preceq P$ we have $i(\Delta_\alpha,\Delta_\beta) = 1$.
\end{lemma}

\begin{proof}
Let $d = \dim \Delta_\beta$. We check each value of $d$ separately. We need to check that $M_{\alpha,\beta} \subset M_\beta$. We will do this by showing that $M_{P,\beta} \subset M_\beta$. Again we fix $0$ as a common vertex of $\Delta_\beta$ and $\Delta_\alpha$. 

 If $d=0$ there is nothing to prove.

If $d=1$ pick the first lattice point along the ray generated by $\Delta_\beta$, starting at $0$. By construction of $A$ this necessarily generates all lattice points of $M$ which are contained in $L(\Delta_\beta)$.

If $d=2$ we do something similar: Pick a pair of primitive lattice points $v,w \in \Delta_\beta$ such that the only lattice points of $M$ contained in the set $R_{v,w}= \{ av + bw | 0 \leq a,b \leq 1, a+b \leq 1 \}$ are $0,v,w$. We claim this can always be done.

 Indeed, pick any primitive $v',w'$. Then $R_{v',w'}$ contains finitely many lattice points. If there exists $u \in R_{v',w'}, u \neq 0,v',w'$, we may without loss of generality assume $u$ is primitive, and  consider $R_{u,v'}$ which have fewer lattice points. Iterating this proves the claim.

Now we claim that $v,w$ is a basis for $M \cap L(\Delta_\beta)$. If not, then by Lemma \ref{basis} there is a lattice point $p \in M$ such that $p = av +bw$ with $0 \leq a,b < 1$. By assumption $a+b > 1$. But then $v+w-p = v(1-a) + w(1-b)$ is a lattice point in $R_{v,w}$ different from $0,v,w$ which is a contradiction.
\end{proof}

Assuming the polytope $P$ is very ample, we have that $X_{P \cap M} \simeq X_{\Sigma_P}$. In this case it will be convenient to be able to compute the local Euler obstruction using the language of fans (for instance when we relate it to the resolution of singularities for surfaces), so we describe how this is done.

We have the identification, for a vertex $v$ of $P$, of  $C_v = \Cone(P \cap M - v)$ with a cone $\sigma^\vee \subset M_\R$ dual to a maximal cone $\sigma$ in the normal fan $\Sigma_P$. This is compatible with face inclusions: If $\Delta_\alpha$ is a face of $P$ containing $v$, there is a corresponding face $\tau_\alpha$ of $\sigma$. We then have $\RSV_\Z(\Delta_\alpha,\Delta_\beta) = \RSV_\Z(\tau_\alpha^\vee,\tau_\beta^\vee)$ where the last expression means:

Let $M_\beta' = M_\R / L(\tau_\beta^\vee)$ and let $K_{\alpha,\beta}$ be the image of $\tau_\alpha^\vee$ in $M_\beta'$. Then $\RSV(\tau_\alpha^\vee,\tau_\beta^\vee)$ equals $\Vol (K_{\alpha,\beta} \setminus \Theta_{\alpha,\beta})$ where $\Theta_{\alpha,\beta}$ is $\Conv (K_{\alpha,\beta} \cap M_\beta' \setminus \{ 0 \})$, and the volume is normalized with respect to the lattice $M_\beta' \cap L(K_{\alpha,\beta})$.

\section{Weighted projective spaces}

Our main examples in this paper are the weighted projective spaces(wps), which are defined  as follows:

Let $q_0,...,q_n \in \mathbb{N}$ satisfy gcd$(q_0,...,q_n) = 1$. Define $\mathbb{P}(q_0,...,q_n) = (\mathbb{C}^{n+1} \setminus \{ 0 \}) / \sim$ where $\sim$ is the equivalence relation:
\[(a_0,...a_n) \sim (b_0,...,b_n) \ekv a_i = \lambda^{q_i}b_i  \text{ for all i, for some } \lambda \in \mathbb{C}^{\ast}. \]
We call $\mathbb{P}(q_0,...,q_n)$ the wps corresponding to  $q_0,...,q_n$. Observe that $\p(1,...,1) \simeq \p^n$. We can construct a wps as a toric variety by the following:

Given natural numbers $q_0,...,q_n$ with $\gcd(q_0,...,q_n)=1$, consider the quotient lattice $\Z^{n+1}$ by the subgroup generated by $(q_0,...,q_n)$, and write $N=\mathbb{Z}^{n+1} / \mathbb{Z} (q_0,...,q_n)$. Let $u_i$ for $i=0,...,n$ be the images in $N$ of the standard basis vectors of $\Z^{n+1}$. This means that in $N$ we have the relation
\[ q_0u_0 + ... + q_nu_n = 0 . \]
Let $\Sigma$ be the fan consisting of all cones generated by proper subsets of $\{ u_0,...,u_n \}$. Then $X_\Sigma = \mathbb{P}(q_0,...,q_n)$. By the quotient construction of toric varieties one gets by \cite[Example 5.1.14]{Cox} that $X_\Sigma$ is a geometric  quotient, whose points agree with the set-theoretic definition given above.

From \cite{RossiTerra} we can also describe the wps as embedded in projective space via a polytope $P$ giving $\p(q_0,...,q_n) \simeq X_P$:

Given $(q_0,...,q_n)$ and $M \cong \Z^{n+1}$, let $\delta = \lcm(q_0,...,q_n)$. Consider the $n+1$ points of $M_\R \cong \R^{n+1}$:
\[ v_i = (0,...,\frac{\delta}{q_i},...0), i = 0,...,n .\]
Let $\Delta$ be the convex hull of $0$ and all $v_i$. Intersecting $\Delta$ with the hyperplane $H= \{ (x_0,...,x_n) | \sum_{i=0}^n x_iq_i = \delta \}$, we get a $n$-dimensional polytope $P$. Then $X_P \cong \p(q_0,...,q_n)$ and the associated divisor $D_P$ will be  $\frac {\delta}{q_0}D_0$.  This divisor is very ample and its class generates $\Pic(\mathbb{P}(q_0,...,q_n)) \simeq \Z$. When we speak of the degree of the dual variety of a weighted projective space, we will always mean using the embedding given by $D_P$.

There are characterizations of when $\mathbb{P}(q_0,...,q_n) \simeq \mathbb{P}(s_0,...,s_n)$ in terms of the weights, see for instance \cite{RossiTerra}. The upshot is that we can  assume the weights are reduced, i.e., that for all $i$ $\GCD(q_0,...,\ol{q_i},...q_n)=1$. We will always make this assumption.

Following  \cite[5.15]{Fletcher} we can describe the singular locus of the wps: Recall that the fan $\Sigma$ is the collection of cones $\Cone(u_j | j \in J)$ for all proper subsets $J \subset \{0,...,n\}$. Set $\sigma_{j_1,...,j_k}=\Cone(u_{j_1},...,u_{j_k})$. Fixing one such cone $\sigma_{j_1,...,j_k}$, let $I=\{ i_0,...,i_{n-k} \} = \{0,...,n \} \setminus \{j_1,...,j_k\}$. Then we have:
\begin{proposition} \label{singlocus} \cite[Prop 2.1.7]{biun}
 $\p(q_0,...,q_n)$ is nonsingular in codimension $k$ if for all $\{j_1,...,j_k \}$, the corresponding $\gcd(q_{i_0},...,q_{i_{n-k}})=1$. In particular:

$\p(q_0,...,q_n)$ is nonsingular in codimension $1$ .

$\p(q_0,...,q_n)$ has isolated singularities if and only if $\gcd(q_i,q_j)=1$ for all $i,j$.
\end{proposition}

Thus for  surfaces we will always have isolated singularities, but in larger dimensions we might have larger singular locus, for instance $\p(2,2,3,3)$ does not have isolated singularities.

\section{The surface case}

In this section we will let $A$ consist of all lattice points of a $2$-dimensional lattice polytope $P$. Recall that then we have $X_A = X_{P \cap M} \simeq X_P \simeq X_{\Sigma_P}$ and $X_A$ is normal. From Proposition~\ref{dualformula} we have
\[ \deg X_P^\vee = 3\Vol(P) - 2 E(P) + \sum_{v \text{ vertex } \in P} \Eu(v), \]
where $E(P)$ is the sum of the normalized lengths of the edges of $P$. Thus we need to compute the Euler obstruction of the singular vertices. By Lemma \ref{index1} all indices $i(\Delta_\alpha,\Delta_\beta)$ are equal to $1$.

By  Corollary \ref{euler} we get for a vertex $v$, letting $e_1,e_2$ be the edges of $P$ containing $v$:
\[\Eu(v) = \RSV_{\mathbb{Z}}(e_1,v)\Eu(e_1) +  \RSV_{\mathbb{Z}}(e_2,v)\Eu(e_2) -  \RSV_{\mathbb{Z}}(P,v),\]
By Lemma \ref{codim}, $\RSV_\Z(P,e_i)=1$ and $\RSV_\Z(e_i,v)=1$, while by Corollary \ref{cod}, $\Eu(e_i)=1$, for $i=1,2$. Thus we reduce calculations to:
\[ \Eu(v) = 2 - \RSV_{\mathbb{Z}}(P,v). \]
To calculate $ \RSV_{\mathbb{Z}}(P,v)$ we get that $M_P/M_{P,v}$ will  equal $M$. Hence $S_{P,v}$ will  be the semigroup generated by the lattice points of the polytope $P$, after translating $P$ such that $v$ is the origin. Then  $\RSV_{\mathbb{Z}}(P,v)$ will be the area removed, if we instead of $P$ consider the convex hull of the points of $(P  \setminus \{v \}) \cap M$.
\begin{lemma} \label{euob}
For a $2$-dimensional lattice polytope $P$ and a vertex $v$ we have
\[ \Eu(v) = 1 -c, \]
where $c$ is the number of internal lattice points of $P$ which are boundary points of $\Conv((P \setminus v) \cap M)$.
\end{lemma}
\begin{proof}
By the above discussion
\[\Eu(v) = 2 - \Vol(P) + \Vol (\Conv((P \setminus v) \cap M)) .\]
(This formula is also found in \cite[Corollary 3.2]{SturmHelm}, \cite[Proposition 5.2.12]{Mork},\cite[Proposition 1.11.7]{biun}.) By Pick's formula 
\[ \Vol(P) = 2i +b-2, \]
\[ \Vol(\Conv((P \setminus v) \cap M)) = 2(i-c)+(b+c-1) -2, \]
hence 
\[  \Eu(v) = 2 - (1+c)=1-c. \]
\end{proof}

\begin{figure}
\begin{tikzpicture} [scale=0.8]
  \foreach \x in {-1,0,...,4}{
      \node[draw,circle,inner sep=1pt,fill] at (\x,0) {};
    }
  \foreach \y in {-1,0,...,4}{
    \node[draw,circle,inner sep=1pt,fill] at (0,\y) {};
  }
  \node[draw,circle,inner sep=1pt,fill] at (1,3) {};
   \node[draw,circle,inner sep=1pt,fill] at (1,1) {};
   \node[draw,circle,inner sep=1pt,fill] at (1,2) {};
   \node[draw,circle,inner sep=1pt,fill] at (2,1) {};
\draw [->] (-2,0) -- (5,0);
\draw [->] (0,-2) -- (0,5);
\draw [-, ultra thick] (0,0) -- (0,2);
\draw [-, ultra thick] (0,2) -- (1,3);
\draw [-, ultra thick] (1,3) -- (3,0);
\draw [-, ultra thick] (0,0) -- (3,0);
\node [above] at (1,1) {P};
\end{tikzpicture}
\begin{tikzpicture} [scale=0.8]
  \foreach \x in {-1,0,...,4}{
      \node[draw,circle,inner sep=1pt,fill] at (\x,0) {};
    }
  \foreach \y in {-1,0,...,4}{
    \node[draw,circle,inner sep=1pt,fill] at (0,\y) {};
  }
  \node[draw,circle,inner sep=1pt,fill] at (1,2) {};
  \node[draw,circle,inner sep=1pt,fill] at (1,1) {};
   \node[draw,circle,inner sep=1pt,fill] at (1,2) {};
   \node[draw,circle,inner sep=1pt,fill] at (2,1) {};
\draw [->] (-2,0) -- (5,0);
\draw [->] (0,-2) -- (0,5);
\draw [-, ultra thick] (0,0) -- (0,2);
\draw [-, ultra thick] (0,2) -- (1,2);
\draw [-, ultra thick] (1,2) -- (3,0);
\draw [-, ultra thick] (0,0) -- (3,0);
\end{tikzpicture}

\caption{The polytope $P=\Conv((0,0),(0,2),(1,3),(3,0))$. Removing the vertex $(1,3)$ we get the right figure. $\Vol(P) = 11$ while the volume of the new polytope is $8$. Hence $\Eu(1,3)=2-11+8=-1$. }
\end{figure}
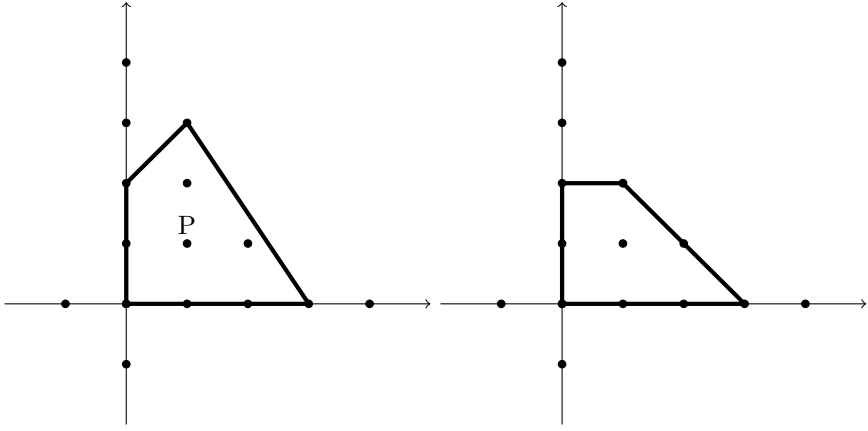
One can also describe the Euler obstruction in terms of a resolution of singularities:
\begin{proposition} \cite{GGS} \label{freu}
Let $p \in S$ be a normal cyclic surface singularity, and $X \rightarrow S$ a minimal resolution of the singularity $p$ with exceptional curves $E_i$. Then
\[ \Eu(p) = \sum_i (2 + E_i \cdot E_i ). \]
\end{proposition}
We will relate these two descriptions of the Euler obstruction.

One can describe resolutions of singularities for toric varieties in general (see for instance \cite[Ch. 11.1]{Cox}), and for surfaces the minimal resolution can be made quite explicit (we follow descriptions in  \cite{Fractions} and \cite{Dais1}). Now we switch to the language of fans.

Given a rational number $\lambda$, we can consider the Hirzebruch--Jung (HJ) continued fraction
\[ \lambda = b_1 - \frac {1} {b_2 - \frac{1} {... - \frac{1} {b_r} }},  \]
which we will denote by $[b_1,...,b_r]^-$.

Since this is a local computation, we do this cone by cone, so we assume $\sigma$ is a $2$-dimensional cone. We include the proof of the following result, which is well-known, because it shows how to construct the integers $k$ and $d$:
\begin{lemma} \label{cone} 
Given any singular $2$-dimensional cone $\sigma$, one can choose a basis $\{ e_1,e_2 \}$ for the lattice $L$ such that in this basis $\sigma = $ Cone$(e_1,ke_1+de_2)$, where $d>k>0$ and $\gcd(d,k)=1$.
\end{lemma}
\begin{proof}
We can always choose a primitive generator $u$ of an edge of $\sigma$ as the first basis vector of our lattice. Let $(e_1=u,e_2')$ be a basis for the lattice. The other facet of the cone will in this basis be generated by a  vector $w = ae_1+be_2'$.  Now let $d=|b|$ and $k=a \mod{d}$, where $0 < k < d$.

Then $w =(a-k+k)e_1+ \sign(b) d e_2' = ke_1+d(\sign(b)e_2' + \frac {a-k} {d}e_1)$. Thus we see that in the new basis $\{e_1,e_2=\sign(b)e_2' + \frac {a-k} {d}e_1 \}$, $w =  ke_1+de_2$.
\end{proof}
\begin{definition}
We say that a cone $\sigma$ is of type $(d,k)$ if it can be written as in Proposition~\ref{cone} with parameters $d,k$. 

Note also that some literature, notably \cite{Cox} and \cite{Fulton}, use a different convention for a $(d,k)$-cone, so that some results sometimes appear a bit different. 
\end{definition}
\begin{lemma} \label{dk-cone} \cite[Lemma 3.3]{Dais1}
Assume $\sigma^\vee$ is a $(d,k)$-cone in $M_\R$ with respect to $\{ e_1,e_2 \}$. Then $\sigma$ is a $(d,d-k)$-cone in $N_\R$ with respect to the basis $\{e_2^*, e_1^*-e_2^* \}$.
\end{lemma}
\begin{construction} \label{conres} \cite[Section 4]{Fractions}
Set $K(\sigma) =  \Conv (\sigma \cap (N \setminus \{ 0 \} ))$. Let $P(\sigma)$ be the boundary of $K(\sigma)$ and $V(\sigma)$ the set of vertices. $P(\sigma)$ is a connected polygonal line with endpoints coinciding with the generators of $\sigma$.

Let the primitive generators of $\sigma$ be $v_1,v_2$. Let $A_0 = v_1$. Define $A_i$, $i \ge 0$ as the sequence of lattice points as one goes along the enumerated edges of $P(\sigma)$. This is a finite sequence and the last point $v_2$ is denoted by $A_{r+1}$.

By construction  each pair $(A_i,A_{i+1})$ is a basis for $N$, since the triangle formed by $0, A_i, A_{i+1}$ has no other lattice points. Also the slopes of the set $\{ A_i \}$ have to increase with increasing $i$, since $A_i$ are on the boundary of a convex set. Thus we have relations:
\[ rA_{i-1} + sA_i = A_{i+1}, \]
\[ tA_i + uA_{i+1} = A_{i-1}, \]
which implies
\[ (rt+s)A_i + (ru-1)A_{i+1} = 0, \]
\[ rt+s = 0, \: ru = 1. \]
If $r=u=1$ we get $s=-t$ and
\[ sA_i + A_{i-1}=A_{i+1}. \]
But this contradicts the increasing of the slopes. Thus we must have $r=u=-1$ and $s=t$, resulting in the relation
\[ A_{i-1} + A_{i+1} = b_iA_i .\]
By convexity we must have $b_i \geq 2$.
\end{construction}

\begin{proposition} \cite[Prop. 4.3]{Fractions} \label{contFr}
By   Construction \ref{conres} for a $(d,k)$-cone $\sigma$, we get that $[b_1...,b_r]^- = \frac {d} {d-k}$.
\end{proposition}

\begin{example}
 In Figure \ref{83}  we see  Construction \ref{conres} for $(d,k)=(8,3)$. The lattice points $A_i$ are the following: 
\[ A_0 =  \begin{bmatrix} 1 \\ 0 \end{bmatrix}, \hspace{6 pt} A_1 =  \begin{bmatrix} 1 \\ 1 \end{bmatrix}, \hspace{6 pt} A_2 =  \begin{bmatrix} 1 \\ 2 \end{bmatrix}, \hspace{6 pt} A_3 =  \begin{bmatrix} 2 \\ 5 \end{bmatrix}, \hspace{6 pt} A_4 =  \begin{bmatrix} 3 \\ 8 \end{bmatrix}. \]
The continued fraction $\frac {d} {d-k}=\frac {8} {5}$ equals $[2,3,2]^{-1}$. By Proposition \ref{contFr} this  is equivalent to the fact that 
\[ A_0 + A_2 = 2 A_1, \hspace{6 pt}  A_1 + A_3 = 3 A_2, \hspace{6 pt}  A_2 + A_4 = 2 A_3. \] 
\end{example}
\begin{figure} 
\begin{tikzpicture} [scale=0.9]
  \foreach \x in {-1,0,...,3}{
      \node[draw,circle,inner sep=1pt,fill] at (\x,0) {};
    }
  \foreach \y in {-1,0,...,8}{
    \node[draw,circle,inner sep=1pt,fill] at (0,\y) {};
  }
 \node[draw,circle,inner sep=1pt,fill] at (3,8) {};
\draw [->] (-2,0) -- (4,0);
\draw [->] (0,-1) -- (0,9);
\draw [->, ultra thick] (0,0) -- (3,8);
\draw [->, ultra thick] (0,0) -- (1,0);
\node[draw,circle,inner sep=1pt,fill] at (1,1) {};
\node[draw,circle,inner sep=1pt,fill] at (1,2) {};
 \node[draw,circle,inner sep=1pt,fill] at (2,5) {};
 \draw [-, thick] (0,0) -- (1,1);
\draw [-, thick] (0,0) -- (1,2);
\draw [-, thick] (0,0) -- (2,5);
\end{tikzpicture}
\caption{Construction \ref{conres} for (d,k)=(8,3). } \label{83}
\end{figure}
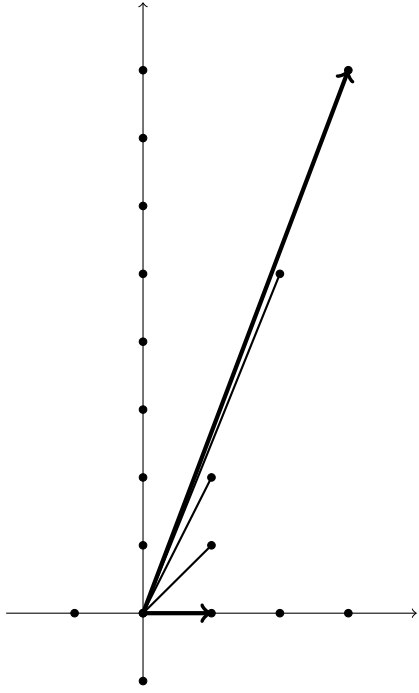
 Given $\sigma$, construct the points $A_i$ as in Construction \ref{conres}. Let $\sigma_i = \Cone(OA_i)$. Let $\Sigma$ be the fan with $2$-dimensional cones $\Cone(\sigma_{i},\sigma_{i-1})$ for $i=0,...,r$. The identity map on the lattice $N$ induces toric morphisms $U_{\sigma_i} \rightarrow U_\sigma$ which glue to a morphism $\phi: X_\Sigma \rightarrow U_\sigma$.
\begin{proposition} \label{res} \cite[Thm. 3.20]{Dais1}
The morphism $\phi$ is a minimal resolution of singularities for $U_\sigma$ with r exceptional components $E_1,...,E_r$ and $E_i^2=-b_i$.
\end{proposition}
By doing this cone by cone, one obtains a global resolution of singularities by glueing the local constructions. Combining Proposition \ref{res} and Proposition \ref{freu} we obtain:
\begin{corollary} \label{bi}
Given a $(d,k)$-cone in $M_{\mathbb{R}}$ (equivalently a $(d,d-k)$-cone in $N_{\mathbb{R}}$), let $v$ be the torus fixed point of $U_\sigma$. Write
\[ \frac {d} {k} = b_1 - \frac {1} {b_2 - \frac{1} {... - \frac{1} {b_r} }} . \]
Then $\Eu(v)=\sum_{i=1}^r (2-b_i)$.
\end{corollary}

We give our own proof of this in the toric case using the formula of Matsui and Takeuchi, without refering to Proposition \ref{freu}. We need a technical lemma:
\begin{lemma} \label{length} \cite[Lemma 1.22]{Oda}
Let $\frac {d}{k} =[b_1,...,b_s]^-$ and $\frac{d}{d-k}=[c_1,...,c_r]^-$. Then
\[ r= 1 + \sum_{i=1}^r (b_i-2). \]
\end{lemma}

\begin{proof}[Proof of Corollary \ref{bi}]
Given any normal toric surface, consider a  vertex $v$. We have that $\Eu(v) =  1 - c$ where $c$ is the number of internal lattice points of $\sigma^\vee$ which are boundary points of $\Conv((\sigma^\vee \setminus \{0\}) \cap M)$. Writing $\frac{d}{d-k}=[c_1,...,c_r]^-$ we have by Construction \ref{conres} and Proposition \ref{contFr} that $c = r$. By Lemma \ref{length} $\Eu(v) = \sum_{i=1}^r (2-b_i)$.
\end{proof}

\begin{remark}
If the cone is smooth, it is isomorphic to $\Cone (e_1,e_2)$, if we by convention set the corresponding continued fraction equal to $[1]^-$, then all formulas for the Euler-obstructions are true also for smooth cones.
\end{remark}

Combining the above we obtain:

\begin{proposition} \label{degdual}
Assume $P$ is a $2$-dimensional lattice polytope. Construct the minimal resolution of singularities of $X_{\Sigma_P}$ and let $E_{v,i}$ be the exceptional divisors for the singularities $v$. Let
\[ \delta = 3\Vol(P) - 2 E(P) + \sum_{v \text{ vertex } \in P} \sum_{i} (2+E_{v,i}^2). \]
Then $X_P^\vee$ is a hypersurface if and only if $\delta$ is non--zero. Assuming $X_P^\vee$ is a hypersurface, it has degree $\delta$.

More explicitly, let $\sigma_1,...,\sigma_r$ be the maximal cones of $\Sigma_P$. Assume $\sigma_i$ is a $(d_i,d_i-k_i)$-cone and write $\frac{d_i}{k_i} = [b_{i,1},...,b_{i,s_i}]^-$. Then
\[ \delta = 3\Vol(P) - 2 E(P) + \sum_{i=1}^r \sum_{j=1}^{s_i} (2-b_{i,j}).  \]
\end{proposition}

We can classify which normal toric surfaces are smooth or Gorenstein  using the Euler obstruction.

\begin{corollary} \cite[Cor. 5.7]{MatTak}
For any point $v$ in a normal toric surface we have that $v$ is smooth if and only if $\Eu(v)=1$.
\end{corollary}
\begin{proof}
This follows directly from Corollary \ref{bi} and the fact that the $b_i$ in Construction \ref{conres} are always $\geq 2$.
\end{proof}

\begin{remark}
Let $A$ be the lattice points from Example \ref{ind1}
\[ \begin{bmatrix} 0 \\ 0 \end{bmatrix} \begin{bmatrix} 0 \\ 1 \end{bmatrix} \begin{bmatrix} 1 \\ 1 \end{bmatrix} \begin{bmatrix} 2 \\ 0 \end{bmatrix} \begin{bmatrix} 2 \\ 1 \end{bmatrix}.  \]
Let $v$ denote the origin and $e_1,e_2$ the edges of $\Conv(A)$ containing $v$. Then we have that $\Eu(v) = i(P,e_1)+i(P,e_2)-\RSV_\Z(P,v) = 2+1-2=1$, even if $v$ corresponds to a singular point of the non-normal variety $X_A$.
\end{remark}

\begin{corollary}
A singular point on a normal toric surface has  Euler-obstruction $0$ if and only if the surface is Gorenstein in a neighbourhood of the point.
\end{corollary}

\begin{proof}

By \cite[Exc. 8.2.13]{Cox} a singular affine toric surface $U_\sigma$ is Gorenstein if and only if $\sigma$ is a $(d,1)$-cone.

Let the singularity be given as a $(d,k)$-cone in $N_{\mathbb{R}}$. Let $\frac {d} {d-k} = [b_1,...,b_r]$. By Corollary \ref{bi} the Euler-obstruction is $0$ if and only if all $b_i=2$. Now if the singularity is Gorenstein, then $k=1$, so $\frac {d} {d-k}=\frac{d}{d-1}$. It is easy to check that the HJ-fraction of $\frac{d} {d-1}$ is a chain of $d-1$ $2$'s.

Conversely if the singularity has Euler-obstruction $0$, then all $b_i$'s are $2$, but by the above this implies that in $M_{\mathbb{R}}$ it is a $(d,d-1)$-cone, so it is a $(d,1)$-cone in $N_{\mathbb{R}}$.
\end{proof}

\begin{remark}
For a surface $X$ the degree of the dual variety given by an embedding by the very ample line bundle $L$ equals the Severi degree $N^{L,1}$.
For a smooth surface one has from \cite{Piene}
\[N^{L,1} = 3 L^2 + 2 L \cdot K_X  +c_2(X),\]
however in the singular case this does not hold. Now fix a toric surface $X_P$. Using Ehrhart theory and Riemann-Roch \cite[Prop. 10.5.6]{Cox} we obtain  that
\[D_P \cdot D_P = \Vol(P), \]
\[-D_P \cdot K_{X_{\Sigma_P}}=E(P). \]
We can combine this with Corollary \ref{degdual} to obtain
\[ N^{D_P,1 }=\deg X_P^\vee = 3D_P^2 + 2D_P \cdot K_{X_P} + \sum_{v} \Eu(v), \]
\[ = 3D_P^2 + 2D_P \cdot K_{X_P} + \sum_{v} \sum_i (2 + E_{v,i}^2), \]
thus $\sum_{v} \Eu(v)$ acts as a sort of ``corrected'' version of $c_2$ for singular surfaces. Indeed, by Remark \ref{ChernMather} $\sum_{v} \Eu(v)$ equals the degree of the second Chern-Mather class $c_2^M(X)$ of the surface. One would have hoped that this correction could work for higher Severi degrees $N^{L,\delta}$, however this seems not to be the case, see for instance \cite{AB}, \cite{Severi}, \cite[Ch. 4]{biun}.
\end{remark}

\subsection{Weighted projective planes}

We wish to apply the results of the previous section to the weigthed projective planes $\p(k,m,n)$ and the $2$-dimensional polytope $P$ defined as the convex hull in $\R^3$ of $v_1=(mn,0,0),v_2=(0,kn,0),v_3=(0,0,km),(0.0,0)$ intersected with the plane $kx+my+nz=kmn$. Denote by $\sigma_i^\vee$ the $2$-dimensional cone generated by the edges of $P$ emanating from $v_i$ (the dual is chosen to remind us that the polytope is in $M$).

\begin{proposition} \label{wpsdeg}
Find  minimal natural numbers $a,b,c$ such that
\begin{align*}
 m+an &\equiv 0 \pmod{k} \\
 n+bk &\equiv 0 \pmod{m} \\
 k+cm &\equiv 0 \pmod{n}
\end{align*}
Then $\sigma_1^\vee$ is a $(k,k-a)$-cone, $\sigma_2^\vee$ is a $(m,m-b)$-cone and $\sigma_3$ is a $(n,n-c)$-cone. 
\end{proposition}
\begin{proof}
We prove this for $\sigma_1^\vee$. $\sigma_1^\vee$ is generated as a cone by the vectors $u_1=(-n,0,k)$ and $u_2=(-m,k,0)$. Picking an $a$ such that $m+an \equiv 0 \pmod{k}$ gives a lattice point of $P$ of the form $v=(d,1,a)$. Then $w=v-v_1$ and $u_1$ is a basis for the lattice spanned by $P$. We have that $u_2 = -a u_1 + k w$, thus $\sigma_1^\vee$ is a $(k,k-a)$-cone.
\end{proof}

\begin{theorem} \label{alg}
Given $\mathbb{P}(k,m,n)$, find  natural numbers $a,b,c$ as in Proposition \ref{wpsdeg}.
Let $\frac {k} {k-a} = [a_1,...,a_t]^-$, $\frac {m} {m-b} = [b_1,...,b_s]^-$, $\frac {n} {n-c} = [c_1,...,c_r]^-$.

Then $\deg \mathbb{P}(k,m,n)^\vee$ equals
\[  3kmn - 2(k+n+m) + \sum_{i=1}^r (2-a_i) + \sum_{i=1}^s (2-b_i) + \sum_{i=1}^t (2-c_i).\]
\end{theorem}

Using Theorem \ref{alg} it is easier to find closed formulas in special cases.

\begin{corollary} \label{2k}
For $k \geq 1$,  $ \degree \mathbb{P}(2k-1,2k,2k+1)^\vee=24k^3-20k+3$.
\end{corollary}

\begin{proof}
We wish to find minimal $a,b,c$ satisfying
\begin{align*}
 2k + a(2k+1) &\equiv 0 \pmod{2k-1}, \\
 2k+1 + b(2k-1) &\equiv 0 \pmod{2k}, \\
 2k-1 + c2k &\equiv 0 \pmod{2k+1} .
\end{align*}
Some easy algebra shows that $a,b,c$ must satisfy
\begin{align*}
 2a &\equiv -1 \pmod{2k-1}, \\
 b &\equiv 1 \pmod{2k}, \\
 c &\equiv -2 \pmod{2k+1} .
\end{align*}
Resulting in $a=k-1, b=1, c=2k-1$. Now 
\begin{align*}
 \frac {2k-1} {2k-1 - (k-1)} &= \frac{2k-1} {k} = [2,k]^- , \\
 \frac {2k} {2k-1} &= [2,...,2]^- ,\\
 \frac {2k+1} {2k+1-(2k-1)} &= \frac {2k+1} {2} = [k+1,2]^- .
\end{align*}
Combining these yields the formula.
\end{proof}

\begin{corollary} \label{fibo}
$\degree \mathbb{P}(m,n,m+n)^\vee = 3mn(m+n) -5(m+n) +4$ .
\end{corollary}

\begin{corollary}
For odd $m>1$, 
\[ \degree \mathbb{P}(m-2,m,m+2)^\vee = 3m^3 -19m+3 .\]
\end{corollary}

\begin{corollary}
$\degree \mathbb{P}(m,n,m+2n)^\vee = 6mn^2+3m^2n-7n- \frac {9} {2} m + \frac {5} {2}$.
\end{corollary}

\begin{proof}
Following Theorem \ref{alg} we want minimal $a,b,c$ such that
\begin{align*}
 n+a(m+2n) \equiv 0 &\pmod{m} ,\\
 mb + m +2n \equiv 0 &\pmod{n} ,\\
 m+cn \equiv 0 &\pmod{m+2n} .
\end{align*}
One sees that $a=\frac {m-1} {2}, b=n-1, c=2$ ($m$ has to be odd, if not then gcd$(m,m+2n) \neq 1$). Now $\frac {m+2n} {m+2n-2} = 2 - \frac {m+2n -4} {m+2n-2} = 2 - \frac {1} {\frac {m+2n-2} {m+2n-4} } = [2,...,2,3]^-$ where the $3$ is by induction, since $\frac {3} {1} = [3]^-$. The Hirzebruch--Jung fraction $\frac {n} {n-(n-1)} = \frac {n} {1} = [n]^-$. Also $\frac {m} {m- \frac {m-1} {2} } = \frac {m} {\frac {m+1} {2}} = [2,\frac {m+1} {2}]^-$. Combining these yields the formula.
\end{proof}

\begin{example}
For sufficiently small examples, these calculations can be doublechecked using Macaulay2\cite{M2}. According to Corollary \ref{fibo} $\degree \p(1,2,3)^\vee = 7$. The lattice points of the polytope defining $\p(1,2,3)$ corresponds to monomials $1,s,s^2,s^3,t,st,t^2$. We run the following code:
\begin{verbatim}
R = ZZ/101[s,t,y1,y2,y3,y4,y5,y6,y7];
f=y1+y2*s+y3*s^2+y4*s^3+y5*t+y6*s*t+y7*t^2;
I=ideal{f,diff(s,f),diff(t,f)};
I =saturate(I,ideal{s*t});
J=eliminate(I,s);
K=eliminate(J,t);
degree K
\end{verbatim}
This outputs the correct answer $7$.
\end{example}

\section{3-folds}

Here we let $A$ be the lattice points of a $3$-dimensional lattice polytope $P$. We have from Proposition \ref{dualformula}:
\[ \deg X_{P \cap M}^\vee = 4\Vol(P) - 3 \sum_{f \lneqq P} \Eu(f) \Vol(f) + 2\sum_{e \lneqq P} \Eu(e)\Vol(e) - \sum_{v \in P} \Eu(v), \]
where $\{ f \}$ is the collection of all facets of $P$, $\{ e \}$ the is collection of all edges of $P$, and the last sum is over all vertices $v$ of $P$.

Again we recall Lemma \ref{index1} saying that for any two faces $\Delta_\alpha \preceq \Delta_\beta \preceq P$ we have that $i(\Delta_\alpha,\Delta_\beta)=1$. Combining this with Corollary \ref{cod}   we see that $\Eu(f) = 1$ for any facet $f$ of $P$.

For an edge $e$ of $P$ we have by \ref{euler} that 
\[ \Eu(e) = - \RSV_\Z(P,e) \Eu(P) + \sum_{e \preceq f, \dim f = 2} \RSV_\Z(f,e) \Eu(f_i) = - \RSV_\Z(P,e) + f_e , \]
where $f_e$ is the number of facets of $P$ containing $e$.

By unraveling the definition of $\RSV_\Z$ we see that the term $\RSV_\Z(P,e)$ is nothing but $\Vol (\ol{P} \setminus\Conv((\ol{P} \setminus \ol{e}) \cap \ol{M}))$, where $\ol{M}$ is the quotient $M/\Z e$ and  $\ol{P}$, $\ol{e}$ are the images of $P$ and $e$ in $\ol{M}$ (Note that $\ol{e}$ is the origin of $\ol{M}$ and will be a vertex of $\ol{P}$).  But this we can calculate: Write the $2$-dimensional cone generated by $\ol{P}$ with apex $\ol{e}$ as a $(d,k)$-cone and write $\frac{d}{k} = [b_1,...,b_r]^-$. Then $\RSV(P,e) = 2 + \sum_{i=1}^r (b_i-2)$ by the arguments in the surface case. Summing up we get
\[ \Eu(e) = f_e -2 + \sum_{i=1}^r (2-b_i). \]
For a vertex $v$ of $P$ we have
\[ \Eu(v) = \Eu(P)\RSV_\Z(P,v) - \sum_{i} \Eu(f_i)\RSV_\Z(f_i,v) + \sum_{j} \Eu(e_j)\RSV_\Z(e_j,v) \]
\[ =  \RSV_\Z(P,v) - \sum_{v \preceq f, \dim f = 2} \RSV_\Z(f,v) + \sum_{v \preceq e, \dim e = 1} \Eu(e) .\]
Calculating most of these terms are easy, $\Eu(e)$ we did above, while similarly to before $\RSV_\Z(f,v) = 2 + \sum_{i=1}^s (c_i-2 )$, where the cone spanned by $f$ with apex $v$ is a $(d,k)$-cone with $\frac{d}{k} = [c_1,...,c_s]^-$. The remaining term $\RSV_\Z(P,v)$, however, is problematic, we need to compute the $3$-dimensional $\Vol(P \setminus\Conv((P \setminus v) \cap M))$. There seems to be no known general method for doing this. However for sufficiently small polytopes, computer programs cabable of calculating convex hulls and volumes can do this, for instance Macaulay2.  Collecting the above we get:

\begin{algorithm} \label{3foldalg}
To calculate the degree of the dual variety of a toric $3$-fold $X_{P \cap M}$, do the following:
\begin{enumerate}
\item Calculate the volume $V$ of $P$.
\item Calculate the sum of the areas of facets of $P$, denoted $A$.
\item For each edge $e$ calculate the length of $e$, denoted $L(e)$.
\item For each edge $e$, let $\sigma_e$ be the cone generated by $P$ with apex $e$ in $M/e \Z$. Write $\sigma_e$ as a $(d,k)$-cone, and write $\frac{d}{k} = [b_1,...,b_r]^-$. Then $\Eu(e)=f_e -2 + \sum_{i=1}^r (2-b_i)$.
\item For each vertex calculate $\RSV_\Z(P,v)$.
\item For each pair consisting of a vertex $v$ and a facet $f$ containing it, write the cone generated by edges of $f$ emanating from $v$ as a $(d_f,k_f)$-cone and write $\frac{d_f}{k_f} = [c_{f,1},...,c_{f,s}]^-$. Then $\RSV_\Z(f,v)=2 + \sum_{i=1}^s (c_{f,i}-2 )$.
\item For each vertex $v$ calculate 
\[ \Eu(v) = \RSV_\Z(P,v) - \sum_f [2 + \sum_{i=1}^s (c_{f,i}-2 )] + \sum_e \Eu(e), \]
 where the sums are over faces containing $v$.
\end{enumerate}
Then $\deg X_P^\vee = 4V -3A +2 \sum_e \Eu(e) L(e) - \sum_v \Eu(v)$.
\end{algorithm}

\subsection{Weighted projective 3-folds} \label{poly}

We will compute the local Euler obstruction and dual degree for weighted projective spaces of the form $\p(1,k,m,n)$. We may assume $\gcd(k,m,n) = 1$.

Set $d = \lcm(k,m,n)$, and let $P$ be the convex hull in $M_\R$ of $v_0=(0,0,0),v_1=(\frac{d}{k},0,0),v_2=(0,\frac{d}{m},0),v_3=(0,0,\frac{d}{n})$. Then $X_P \simeq \p(1,k,m,n)$.

Since every cone containing $v_0$ is smooth, we only need to calculate for faces containing  $v_1,v_2,v_3$. Thus we will do this for $v_1$, the rest is obtained by cyclic permutation.

Denoting $\gcd(a,b)$ by $(a,b)$, the primitive vectors emanating from $v_1$ are
\[ e_1= \begin{pmatrix} -1 \\ 0 \\ 0 \end{pmatrix}, e_2=\begin{pmatrix} -\frac{n}{(n,k)} \\ 0 \\ \frac{k}{(n,k)} \end{pmatrix}, e_3=\begin{pmatrix} -\frac{m}{(m,k)} \\ \frac{k}{(m,k)} \\ 0 \end{pmatrix} . \]
Let $f_1=\Cone (e_2,e_3),f_2=\Cone (e_1,e_3),f_3=\Cone (e_1,e_2) $.

Then $f_2$ is a $(\frac{k}{(n,k)},n')$-cone where $n' \equiv \frac{n}{(n,k)} \pmod{\frac{k}{(n,k)}}$.

$f_3$ is a  $(\frac{k}{(m,k)},m')$-cone where $m' \equiv \frac{m}{(m,k)} \pmod{\frac{k}{(m,k)}}$.

For $f_1$ we first need to choose a basis for the lattice containing $f_1$:
\begin{lemma} \label{basisFacet}
Pick $a,c$ such that $ak+cn = -m(n,k)$.
Then the vectors
\[  w= \begin{pmatrix} a \\ (n,k) \\ c \end{pmatrix}, \hspace{6 pt} e_2 = \begin{pmatrix} -\frac{n}{(n,k)} \\ 0 \\ \frac{k}{(n,k)} \end{pmatrix}, \]
are a basis for the lattice $M_{f_1}$.
\end{lemma}
\begin{proof}
It is easily verified that $M_{f_1}$ consists of all lattice points $(x,y,z)$ satisfying
\[ kx+my+nz=d, \]
hence $w$ is a vector in $M_{f_1}$. We will apply Lemma \ref{basis} to show that $\{w,e_2\}$ is a basis for $M_{f_1}$.

First we claim that for any $(a,b,c)$ in $M_{f_1}$ we must have
\[ b \equiv 0 \pmod{(n,k)}. \]
Indeed, $mb = d-ak-cn$ is congruent to $0$ modulo $(n,k)$, and since $\gcd(k,m,n)=1$ we must have $b$ congruent to $0$ modulo $(n,k)$.

Assume now that $sw + te_2$, $0 \leq s,t < 1$ is a point in $M_{f_1}$. By the above claim we must have $s=0$. But then also $t=0$, hence we are done.
\end{proof}

It will be convenient to choose a particular basis corresponding to the pair $(a,c)$ from Lemma  \ref{basisFacet}, hence we require that $c$ is the minimal non--negative number satisfying $ak+cn = -m(n,k)$, for some $a$. Dividing by $(k,n)(k,m)(m,n)$ and considering this$\pmod{\frac{k}{(n,k)(m,k)}}$ it is clear that this $c$ satisfies $c(n,k) < k$. Then since
\[\begin{pmatrix} -\frac{m}{(m,k)} \\ \frac{k}{(m,k)} \\ 0 \end{pmatrix}=-\frac{c}{(m,k)}\begin{pmatrix} -\frac{n}{(n,k)} \\ 0 \\ \frac{k}{(n,k)} \end{pmatrix} +\frac{k}{(m,k)(n,k)} \begin{pmatrix} a \\ (n,k) \\ c \end{pmatrix},   \]
 and $0 < k-c(n,k) < k$, $f_1$ is a $(\frac{k}{(m,k)(n,k)},\frac{k-c(n,k)}{(m,k)(n,k)})$-cone. From this we can compute the terms $\RSV_\Z(f_i,v_1)$  using HJ-fractions.

For the Euler-obstruction of the edges, we have $\Eu(e_1)=1$ since the cone generated by the image of the two other vectors in $\Z^3/e_1\Z$ is smooth.

To calculate $Eu(e_2)$, set $a=\frac{n}{(n,k)}, b=\frac{k}{(n,k)}$. Choose integers such that $ea+fb=1$. Then the following will be a basis for $\Z^3$:
\[ v_1 = \begin{pmatrix} 0 \\ 1 \\ 1 \end{pmatrix},v_2= \begin{pmatrix} -f \\ e \\ 0 \end{pmatrix}, v_3=\begin{pmatrix} -a \\ 0 \\ b \end{pmatrix}, \]
Since $e_1=bev_1+bv_2+ev_3$, the image in the quotient lattice $\Z^3/e_2$ is $(be,b)$. 

Setting $c=\frac{m}{(m,k)}, d=\frac{k}{(m,k)}$, we have $ e_2=(fbd-bcd)v_1 + (ad+bc)v_2 + (ce-fd)v_3$. In the quotient this is $(fbd-bce,ad+bc)$.

Writing out the details and cancelling common factors (to get primitive vectors) we get that the cone with apex $0$ generated by the image of $P$ is $\Cone((fk-em, n+m),(e,1))$. Now since 
\[ \begin{pmatrix} fk-em \\ n+m \end{pmatrix} = (n+m) \begin{pmatrix} e \\ 1 \end{pmatrix} + (n,k) \begin{pmatrix} 1 \\ 0 \end{pmatrix}, \]
 we get a $((n,k),m \pmod{(n,k)})-$cone.

Similarly for $\Eu(e_3)$ we get a $((m,k),n \pmod{(m,k)})$-cone. Using this and HJ-fractions we can compute the terms $\Eu(e_i)$. 

\begin{example} \label{61015}
We will apply the above to $\mathbb{P}(1,6,10,15)$. Then $v_0=(0,0,0),v_1=(5,0,0),v_2=(0,3,0),v_3=(0,0,2)$. We will do all the steps of Algorithm \ref{3foldalg}.

We calculate that $V(P)=30$ and that $A(P)=1+15+10+6=32$.

Denote the edge connecting $v_i$ and $v_j$ by $e_{ij}$. Denote the facets containing $v_i,v_j,v_k$ by $f_{ijk}$. Then
\[ L(e_{01})=5,L(e_{02})=3,L(e_{03})=2,L(e_{12})=1,L(e_{1,3})=1,L(e_{2,3})=1. \]

Applying the discussion above we can further conclude the following
\[ \Eu(e_{0i})=1 \text{ for } i=1,2,3, \Eu(e_{12})=0, \Eu(e_{13})=-1, \Eu(e_{23})=-3. \]

We know that $\RSV(P,v_0)=1$. Using Macaulay2 we calculate that 
\[ \RSV_\Z(P,v_1)=4, \RSV_\Z (P,v_2)=6, \RSV(P,v_2)=7. \]

For a fixed vertex $v_i$ and a facet $f$ containing it we will then write the cone with apex $v_i$ generated by edges of the facet $f$ as a $(d,k)$-cone.
\begin{equation*}
\begin{array}{llll}
  \text{vertex} & \text{facet} &  \text{corresponding } (d,k)  \\
  \hline
v_1 & f_{123} & (1,0) \\ v_1 & f_{012} & (2,1) \\ v_1 & f_{013} & (3,2) \\
v_2 & f_{123} & (1,0) \\ v_2 & f_{012} & (2,1) \\ v_2 & f_{023} & (5,3) \\
v_3 & f_{123} & (1,0) \\ v_3 & f_{013} & (3,1) \\ v_3 & f_{023} & (5,2) \\
\end{array}
\end{equation*}
Then we have that
\[ \Eu(v_0)=1, \Eu(v_1)=-1, \Eu(v_2)= -2, \Eu(v_3)=-2. \]
Thus we can in turn conclude that 
\[ \deg \p(1,6,10,15)^\vee = 4 \cdot 30 - 3 \cdot 32 + 2(5+3+2-4)-(1-1-2-2)=40. \]
\end{example}

\subsection{Isolated singularities}

If we assume the variety $X_{P \cap M}$ has only isolated singularities we know that $\Eu(e) = 1$ for every edge. Thus we can reduce to
\[ \deg X_{P \cap M}^\vee = 4\Vol(P) - 3A(P) + 2  E(P) - \sum_{v \in P} \Eu(v), \]
where $A(P)$ is the sum of areas of facets of $P$, while $E(P)$ is the sum of lengths of edges of $P$. For a singular point $v$ associated to a vertex of $P$
\begin{equation} \label{isoEu} \Eu(v)=  \RSV_\Z(P,v)  - \sum_{v \preceq f, \dim f = 2} \RSV_\Z(f,v) + e, \end{equation}
where $e$ is the number of edges of $P$ containing $v$.

We need a generalization of Pick's formula to estimate the volume $\RSV_\Z(P,v)$. To do this we make the following definitions:

\begin{definition}
A piecewise linear lattice polygon (pllp) $K$ is a union  $\cup_{i=1}^n K_i$ of some facets of a 3-dimensional convex lattice polytope $P$ which is contractible and connected in codimension one,  meaning that for any pair $K_i,K_j$ there is a chain $K_i = K_{l_1},...,K_{l_s}=K_j$ such that $K_{l_r}$ and $K_{l_{r+1}}$ has $1$-dimensional intersection, for $1 \leq r \leq s-1$.

A lattice point $x$ in $K$ is a boundary point if it is also contained in some facet $F$ of $P$ which is not contained in $K$. If $x$ is not a boundary lattice point, then it is an internal lattice point.
\end{definition}

\begin{proposition}[Generalized Pick's formula]
For a pllp $K = \cup_{i=1}^n K_i$, let $K_i$ be contained in the plane $H_i$. Let $A_i$ be the area of $K_i$, normalized with respect to the lattice generated by lattice points in $H_i$. Then the normalized area of $K$, defined as $A_K \defeq \sum_{i=1}^n A_i$, equals $2i+b-2$, where $b$ is the number of boundary lattice points, and $i$ is the number of internal lattice points.
\end{proposition}

\begin{proof}
We do induction on $n$. If $n=1$ this is just the usual Pick's formula in the plane. Assume we have showed the proposition for $n-1$, and let $K=\cup_{i=1}^n K_i$. We have $A_K = A_{K_n} + A_{K'}$ where $K'=\cup_{i=1}^{n-1} K_i$. Without loss of generality we may assume that we have chosen $K_n$ such that $K'$ is a pllp. Let $i',b'$ be the internal and boundary lattice points of $K'$ respectively. By the inductive hypothesis we have
\[ A_{K'} = 2i' + b' - 2, \]
and by Pick's formula in the plane we have
\[ A_{K_n} = 2i_n + b_n -2, \]
where $i_n, b_n$ are internal and boundary lattice points of $K_n$.
Now we have to compute $i$ and $b$. The boundary points of $K'$ which intersect $K_n$ either are internal in $K$ (call the number of such $k$) or remain boundary points in $K$ (call the number of such $s$). If we let $l$ be the number of boundary points of $K_n$ not in any $K_i$, $i \neq n$, then we have 
\begin{align*} b &= b' -k +l, \\
i &= i' +i_n +k . \end{align*}
Then we get
\[ 2i +b -2 = 2i' + 2i_n + 2k +b' -k + l -2 = A_{K'} + 2i_n + k + l. \]
Thus if we can show that $b_n-2 = k+l$ we are done. By construction $b_n = k+l+s$, hence we need to show that $s=2$:

Consider the set $S=P \setminus K$ where $P$ is the ambient polytope $\Conv(K)$. If $S$ is  nonempty and not connected, then it is clear that $K$ cannot be contractible. Thus we have that $S$ is connected. Then the boundary of $K$ is  $S$ intersected with $K$, which again has to be connected. Now if $s>2$ we have that the boundary of $K$ intersected with $K_n$ cannot be connected. But this implies that the boundary of $K'$ cannot be connected, which contradicts it being a pllp.
\end{proof}

When attempting to compute $\RSV_\Z(\sigma^\vee,v)$ for the vertex of a $3$-dimensional cone $\sigma$, there naturally arises a pllp : Let $K$ be the union of the compact faces of the convex hull of the set $ (\sigma^\vee \setminus \{ v \}) \cap M$. It is a pllp whose ambient polytope is the convex hull of $K$.

\begin{proposition}
For an isolated singular point $v$ on a toric $3$-fold $X_{P \cap M}$ we always have $\Eu(v) \geq 1$.
\end{proposition}
\begin{proof}
Consider the cone $\sigma^\vee$ generated by $P$ with apex $v$. Let $e$ be he number of rays of $\sigma^\vee$ which is always $\geq 3$. Let $K$ be the  pllp associated to $\sigma^\vee$. By Construction \ref{conres} we see that $\sum_{v \preceq f, \dim f = 2} \RSV_\Z(f,v)$ equals the number of boundary points of $K$. By Pick's formula the area of K is $2i+b-2$ where $i$ is the number internal lattice points of $K$. Since $\RSV_Z(P,v) \geq A_K = 2i +b-2$ we get
\[ \Eu(v)=  \RSV_\Z(P,v) + e - b \geq 2i +b-2+3-b=2i+1 \geq 1. \]
\end{proof}
In Example \ref{61015} we saw that this is not true for non-isolated singularities. Observe also that by the proof the only way one can have $\Eu(v)=1$ is if there are just $3$ edges emanating from $v$.

\begin{corollary}
For an isolated singular point $v$ on a toric $3$-fold $X_{P \cap M}$ one has $\Eu(v) = 1$ if and only if (i) there are exactly $3$ edges emanating from $v$, (ii) the associated pllp $K = \cup_{i=1}^n K_i$ has no internal lattice points, and (iii) for each plane $H_i$ containing $K_i$, the integer distance from $H_i$ to the origin equals $1$.
\end{corollary}

The integer distance of a point $v$ and an integer plane $H$ is the index of the lattice generated by vectors joining $v$ and all integer points of $H$, modulo the lattice $M_P$ generated by lattice points of $P$. See \cite[Rmk. 14.8]{Karpenkov} for details.

We will again compute the local Euler obstruction for a $3$-dimensional wps, now with isolated singularities. This assumption  simplifies some of the calculations. By Proposition \ref{singlocus} $\p(1,k,m,n)$ has isolated singularities if and only if $\gcd(m,n)=\gcd(k,n) = \gcd(k,m)=1$. In this case one can calculate that
\[ \Vol(P) = k^2 m^2 n^2, \]
\[ A(P) = kmn+k^2mn+km^2n+kmn^2, \]
\[ E(P) = k+m+n+mn+kn+km.\]
All this is straigtforward, except for the first term of $A(P)$, but this is  \cite[Prop 3.4]{bezout} for a surface of weights $(k,m,n)$.

The vertex $v_0=(0,0,0)$ is smooth, thus $\Eu(v_0)=1$. Since every vertex is contained in $3$ facets, we get for a vertex $v$
\[ \Eu(v)=  \RSV_\Z(P,v) - 3 + \sum_{v \preceq f, \dim f = 2} (2- c_{f,i}).  \]
For the vertex $v_1=(mn,0,0)$, choose $0 < m',n',s <k$ such that
\begin{align*}
 m' \equiv m \pmod{k}, \\
 n' \equiv n \pmod{k},\\
 m+sn \equiv 0 \pmod{k}. 
\end{align*}
Then the $2$-dimensional cones emanating from $v_1$ are $(k,m'),(k,n'),(k,k-s)$-cones. Using HJ-fractions one can then calculate $\Eu(v_1)$. The rest of the vertices are treated similarly.

\begin{example} \label{counter}
Consider $\mathbb{P}(1,2,3,5)$. The polytope $P$ has vertices $v_0=(0,0,0), v_1=(15,0,0), v_2=(0,10,0), v_3=(0,0,6)$. Using Macaulay2  we calculate that 
\begin{align*}
 RSV_\Z(P,v_1)=4,\\
 RSV_\Z(P,v_2)=5,\\
 RSV_\Z(P,v_3)=6. 
\end{align*}
The cones emanating from $v_1$ are all $(2,1)$-cones, thus all $c_{f,1}=2$, hence $\Eu(v_1)=4-3+0=1$.

For $v_2$ we have $(3,2),(3,2),(3,1)$-cones, giving HJ-fractions $[2,2]^-,[2,2]^-,[3]^-$. Hence $\Eu(v_2)=5-3-1=1$.

For $v_3$ we have $(5,2),(5,3),(5,4)$-cones, giving HJ-fractions $[3,2]^-,[2,3]^-,$ \\ $[2,2,2,2]^-$. Hence $\Eu(v_2)=6-3-1-1=1$. We then get:

\[\deg \mathbb{P}(1,2,3,5)^\vee = 4 \cdot 900-3 \cdot 330+2 \cdot 41-4=2688. \]

\end{example}

\begin{remark}

This example is somewhat surprising, as it exhibits a variety with  isolated singularities which has Euler-obstruction constantly equal to $1$. Matsui and Takeuchi \cite{MatTak} shows that for  normal and projective toric surfaces, the Euler-obstruction is constantly equal to $1$ if and only if the variety is smooth.  They conjectured the similar statement in higher dimensions. This is a counterexample to that conjecture. There are also some other examples, see Appendix \ref{data}.
\end{remark}

In the appendix we list some computations done in Macaulay2 for the local Euler obstructions of weigthed projective $3$-folds. It isn't easy to see a clear pattern. This might be analogous to the computations of the Nash blow-up of toric varieties in \cite{nash}, which in principle could be be used to compute the local Euler obstruction. The authors write ``Almost every straightforward conjecture one might make about the patterns in the Nash resolution seems to be false.''

One would have hoped to be able to compute $\RSV_\Z(P,v)$ for a $3$-dimensional polytope in a way similar to the $2$-dimensional case, for instance using some form of generalized theory of multidimensional continued fractions. However little is still known about this. Karpenkov writes ``... with the number of compact faces greater than 1 almost nothing is known'' \cite[p.219]{Karpenkov}, this corresponds to the number of compact polytopes in the pllp .

\section{Dual defective varieties}

For a variety $X \subset \p^N$, one defines the dual defect $\defect X$ of $X$ to be $\defect X = N-1 - \dim X^\vee$ (i.e., $\defect X = 0$ if and only if $X^\vee$ is a hypersurface in ${\p^N}^\vee$). If $\defect X > 0$ we say that $X$ is defective. Using the theory from the previous sections we give a new proof of the well-known result:

\begin{proposition}
The only normal and projective toric surfaces which are defective are those of the form $\p(1,1,n)$.
\end{proposition}

First we prove an easier result:

\begin{lemma}
The only normal and projective toric surfaces associated to a triangle $P$, which are defective, are those of the form $\p(1,1,n)$.
\end{lemma}

\begin{proof}
We have by \cite[Thm 1.4]{MatTak} that $\defect X > 0$ if and only if the expression
\begin{equation} \label{de} 3\Vol(P) - 2 E(P) + \sum_{v \text{ vertex } \in P} \Eu(v)  \end{equation}
equals $0$. We have that $E(P)=b$, where $b$ is the number of boundary points of $P$, and letting $i$ be the number of internal lattice points, we have by Picks' formula
\[ \Vol(P) = 2i+b-2. \]
We also have that $\Eu(v) = 2 - \RSV_Z(P,v)$. Thus we get
\[ 3 (2i+b-2) -2b + 6 - \sum_{v \text{ vertex } \in P} \RSV_Z(P,v) \]
\[ = 6i+b - \sum_{v \text{ vertex } \in P} \RSV_Z(P,v). \]
We now claim that 
\[ \sum_{v \text{ vertex } \in P} \RSV_Z(P,v)  \leq 3i+b, \]
which would imply that  \eqref{de} is $\geq 0$ with equality only possible if $i=0$.
To see that the claim is true, let $b_1,b_2,b_3$ be the number of lattice points on the $3$ edges of $P$.
By doing Construction \ref{conres} for a vertex we construct a sequence of points $A_0,...,A_{r+1}$. By the construction we see that each of the points $A_1,...,A_r$ has to be either an inner point of $P$ or an inner point of the edge opposite to the vertex. Then we get $r \leq i+ b_j-2$, thus $\RSV_\Z(P,v) = r+1 \leq i+ b_j- 1$, hence 
\[ \sum_{v \text{ vertex } \in P} \RSV_Z(P,v) \leq \sum_{j=1}^3 i +b_j -1 = 3i +b, \]
proving the claim.

If $i=0$, then we need to check when $b = \sum_{v \text{ vertex } \in P} \RSV_Z(P,v)$. Assuming there are two different edges with internal lattice points, we see by Construction \ref{conres} that $\sum_{v \text{ vertex } \in P} \RSV_Z(P,v)=3$. Hence the only way in which a triangle can satisfy
 \[ 3\Vol(P) - 2 E(P) + \sum_{v \text{ vertex } \in P} \Eu(v) = 0,\]
is if it has two edges with no internal lattice points. After a change of basis this will always be a polytope of the form $\Conv ( (0,0),(n,0),(0,1))$ which is isomorphic to $\p(1,1,n)$. That $\defect \p(1,1,n) > 0$ can be easily calculated from Theorem \ref{alg}. Alternatively this also follows from the fact that all cones have positive defect and $\p(1,1,n)$ is the cone over the $n$-th Veronese embedding of $\p^1$, i.e., the rational normal curve of degree $n$.
\end{proof}

Using this we can prove the general case:

\begin{proof}
Let the polytope have vertices $v_1,...,v_n$, indexed such that $v_j$ is connected to $v_{j-1}$ and $v_{j+1}$ via an edge (take indices modulo $n$ when necessary). To estimate $\RSV_\Z(P,v_j)$ we will consider the triangle $T_j \defeq v_{j-1}v_jv_{j+1}$ . Let $i_j$ be the number of internal lattice points of $P$ contained in $T_j$. By a similar argument as in the previous lemma, by Construction \ref{conres} we have that $RSV_\Z(P,v) \leq i_j + 1$. Since an internal vertex of $P$  at most can be contained in two triangles $T_j$, we get that $\sum_{j=1}^n i_j \leq 2i$. Thus
\[ \sum_{v \text{ vertex } \in P} \RSV_Z(P,v) \leq  \sum_{j=1}^n i_j+1 \leq 2i +n .\]
The expression we wish to consider is 
\[ 3\Vol(P) - 2 E(P) + \sum_{v \text{ vertex } \in P} \Eu(v) \]
\[ = 3(2i+b-2) -2b + 2n -  \sum_{v \text{ vertex } \in P} \RSV_Z(P,v) \]
\[ = 6i+b-6+2n -  \sum_{v \text{ vertex } \in P} \RSV_Z(P,v) \geq 6i+b-6+2n -2i -n = 4i+b+n-6 .\]
This last expression is always greater than $0$ when $n > 3$.
\end{proof}

For $3$-folds it is again more difficult to get general results, however for a subclass of wps we can get similar results:

\begin{proposition}
The only defective $3$-dimensional wps of the form $\p(1,k,m,n)$ with only isolated singularities  are those of the form $\p(1,1,1,n)$.
\end{proposition}
\begin{proof}
As before, by \cite[Thm 1.4]{MatTak} for a toric $3$-fold $X$ with isolated singularities, $\defect X > 0$ if and only if the expression
\begin{equation} \label{eqqq} 4\Vol(P) - 3A(P) + 2  E(P) - \sum_{v \in P} \Eu(v)  \end{equation} 
equals $0$. For $\p(1,k,m,n)$ we have as before
\begin{align*}
 \Vol(P) &= k^2 m^2 n^2 ,\\
 A(P) &= knm(1+k+m+n) ,\\
 E(P) &= k+m+n+mn+kn+km ,
\end{align*}
and for a vertex $v$ of $P$  
\[ \Eu(v) =   \RSV_\Z(P,v) + 3 - \sum_{v \preceq f, \dim f = 2} \RSV_\Z(f,v). \]
We now claim that for the vertex $v_1=(mn,0,0)$, $\Eu(v_1) \leq k^2$.

Indeed, by using the description of $P$ from Section \ref{poly} we have that the volume which equals $\RSV_\Z(P,v)$ is enclosed in a polygon with volume
\[ \det \begin{bmatrix} -1 & -n & -m \\ 0  &0  & k \\ 0  & k & 0 \end{bmatrix} = k^2 .\]
Thus $\RSV_\Z(P,v) \leq k^2$. Also  for any face $f$ containing $v$, $\RSV_\Z(f,v) \geq 1$. Combining this we get
\[ \RSV_\Z(P,v) + 3 - \sum_{v \preceq f, \dim f = 2} \RSV_\Z(f,v) \leq k^2 +3 -3 = k^2. \]
By symmetry we also have $\Eu(v_2) \leq m^2, \Eu(v_3) \leq n^2$. Thus \eqref{eqqq} reduces to
\[ 4 k^2 m^2 n^2 - 3knm(1+k+m+n) +2(k+m+n+mn+kn+km) - \sum_{v \in P} \Eu(v) \]
\[ \geq 4 k^2 m^2 n^2 - 3knm(1+k+m+n) +2(k+m+n+mn+kn+km) - 1 -k^2 -m^2-n^2 .\]
If we are not in the case $\p(1,1,1,n)$, we may assume without loss of generality that $k \geq 3, m \geq 2$ and $k > m > n$. We have that
\begin{align*}
 k^2 m^2 n^2 - 3k m^2 n &= k m^2 n (kn-3) \geq 0, \\
 k^2 m^2 n^2 - 3k m n^2 &= k m n^2 (mk-3) \geq 0, \\
 2 k^2 m^2 n^2 -3 k^2 m n - 3kmn - k^2  &= k(k(mn(2mn-3)-1)-3mn).
\end{align*}
Now unless $m=2$ and $n=1$, we have  $2mn-3 \geq 2$, thus $mn(2mn-3)-1 \geq mn$, implying $k(mn(2mn-3)-1) \geq 3mn$. Hence
\[ 2 k^2 m^2 n^2 -3 k^2 m n - 3kmn - k^2  \geq 0 .\]
Also we have that
\begin{align*}
 kn-n^2 &\geq 0, \\
 km - m^2 &\geq 0. 
\end{align*}
Combining all these we get 
\begin{align*}
 4 k^2 m^2 n^2 - 3knm(1+&k+m+n) +2(k+m+n+mn+kn+km) \\
- 1 -k^2 -m^2-n^2 &\geq 2(k+n+m +mn) +kn+km-1 > 0 .
\end{align*}
One can easily verify that the exception $\p(1,k,2,1)$ has defect $0$.

That $\defect \p(1,1,1,n) >0$ follows from the fact that it is the cone over the n-th Veronese embedding of $\p^2$.
\end{proof}

Using our algorithms for  calculations of degrees of dual varieties, we have checked which wps of the form $\p(1,k,m,n)$ that do not necessarily have isolated singularities, are defective. For $k,m,n \leq 10$ we have computed that the only defective wps of the form $\p(1,k,m,n)$ are  $\p(1,1,1,l), \p(1,1,m,lm), \p(1,k,m,km)$ which are cones over $(\p^2,\mathcal{O}(l)), (\p(1,1,m),\mathcal{O}(l)), (\p(1,k,m),\mathcal{O}(1))$ respectively. Based on the numerical data we conjecture the following.

\begin{conjecture}
The only defective wps are those which are cones over a  wps (not necessarily with reduced weights) of lower dimension.
\end{conjecture}

\section*{Acknowledgements}
This article is partly based on the author's master thesis \cite{biun}. I wish to thank my advisor Ragni Piene for many helpful conversations while writing this article. I also wish to thank my co-advisor John Christian Ottem for helpful comments and discussions as well as the anonymous referee for helpful suggestions and comments, in particular for suggesting the current formulation of Lemma \ref{euob}.

\appendix

\section{Computations} \label{data}

The table below shows weights (W), Euler-obstructions (E1,E2,E3) and $\RSV_\Z(P,v)$ (R1,R2,R3) for $\p(1,k,m,n)$, where $k,m,n \leq 10$ and the singularities are isolated. The computations were done using the Macaulay2 package EulerObstructionWPS which can be found at the author's webpage\cite{EBWPSBIUN}. Note also that the package EDPolytope\cite{EDHS} by Helmer and Sturmfels can in principle calculate the degree of the dual variety of toric varieties $X_A$ of any dimension. However unless $A$ is quite small, their computation will not terminate.
\small{
\begin{verbatim}
          W          E1    E2    E3    R1    R2    R3
    _____________    __    __    __    __    __    __
    1     1     1     1     1     1     1     1     1
    1     1     2     1     1     1     1     1     4
    1     1     3     1     1     3     1     1     9
    1     1     4     1     1     7     1     1    16
    1     1     5     1     1    13     1     1    25
    1     1     6     1     1    21     1     1    36
    1     1     7     1     1    31     1     1    49
    1     1     8     1     1    43     1     1    64
    1     1     9     1     1    57     1     1    81
    1     1    10     1     1    73     1     1   100
    1     2     3     1     1     1     1     4     5
    1     2     5     1     1     5     1     4    13
    1     2     7     1     1    13     1     4    25
    1     2     9     1     1    25     1     4    41
    1     3     4     1     3     1     1     9     6
    1     3     5     1     1     3     1     5    11
    1     3     7     1     3     7     1     9    17
    1     3     8     1     1    11     1     5    24
    1     3    10     1     3    19     1     9    34
    1     4     5     1     7     1     1    16     7
    1     4     7     1     1     7     1     6    19
    1     4     9     1     7     9     1    16    21
    1     5     6     1    13     1     1    25     8
    1     5     7     1     5     3     1    13    13
    1     5     8     1     3     5     1    11    16
    1     5     9     1     1    13     1     7    29
    1     6     7     1    21     1     1    36     9
    1     7     8     1    31     1     1    49    10
    1     7     9     1    13     3     1    25    15
    1     7    10     1     7     7     1    17    22
    1     8     9     1    43     1     1    64    11
    1     9    10     1    57     1     1    81    12
    2     3     5     1     1     1     4     5     6
    2     3     7     1     1     3     4     5    10
    2     5     7     1     3     1     4    11     7
    2     5     9     1     1     5     4     6    15
    2     7     9     1     7     1     4    19     8
    3     4     5     1     1     1     5     6     6
    3     4     7     3     1     1     9     6     7
    3     5     7     1     1     3     5     6    10
    3     5     8     1     5     1     5    13     7
    3     7     8     1     7     1     5    17     7
    3     7    10     3     3     1     9    13     8
    4     5     7     1     1     3     6     6    10
    4     5     9     7     1     1    16     7     8
    4     7     9     1     3     1     6    12     7
    5     6     7     5     1     1    13     8     7
    5     7     8     1     3     1     6    13     7
    5     7     9     1     1     5     6     7    15
    5     8     9     1     5     1     6    16     8
    7     8     9    13     1     1    25    10     8
    7     9    10     3     3     1    10    15     8
\end{verbatim}}

The table below shows weights (W), Euler-obstructions (E1,E2,E3) and $\RSV_\Z(P,v)$ (R1,R2,R3) for $\p(1,k,m,n)$, where $k,m,n \leq 6$, where the singularities are not  isolated.
\small{
\begin{verbatim}
         W         E1    E2    E3    R1    R2    R3
    ___________    __    __    __    __    __    __

    1    2    2     1     0     0     1     2     2
    1    2    4     1     0     2     1     2     8
    1    2    6     1     0     8     1     2    18
    1    3    3     1    -1    -1     1     3     3
    1    3    6     1    -1     3     1     3    12
    1    4    4     1    -2    -2     1     4     4
    1    4    6     1     2     2     1     8    10
    1    5    5     1    -3    -3     1     5     5
    1    6    6     1    -4    -4     1     6     6
    2    2    3     0     0     1     2     2     5
    2    2    5     0     0     3     2     2    11
    2    3    3     1     0     0     4     2     2
    2    3    4     0     1     0     2     5     4
    2    3    6     0     0     1     2     2     6
    2    4    5     0     2     1     2     8     6
    2    5    5     1    -1    -1     4     3     3
    2    5    6     0     5     0     2    13     5
    3    3    4    -1    -1     1     3     3     6
    3    3    5     0     0     5     2     2    13
    3    4    4     3     0     0     9     2     2
    3    4    6    -1     0    -1     3     4     4
    3    5    5     1    -1    -1     5     3     3
    3    5    6     0     3     1     2    11     4
    4    4    5    -2    -2     1     4     4     7
    4    5    5     7     0     0    16     2     2
    4    5    6     2     1     2     8     7     5
    5    5    6    -3    -3     1     5     5     8
    5    6    6    13     0     0    25     2     2
\end{verbatim}}

\bibliography{ref}
\bibliographystyle{plain}
\end{document}